\documentclass[12pt]{amsart}
\usepackage{amsfonts}
\usepackage{amssymb}
\usepackage[cmyk]{xcolor}

\newcount\minleft
\newcount\timehour
\def\thetime{\timehour=\time
\divide\timehour by60 \minleft=\timehour \multiply\minleft by -60
\advance\minleft by\time \ifnum\time>720\advance\timehour
by-12\fi\relax
\number\timehour:\ifnum\minleft<10 %
    0\fi\relax\number\minleft
    \ifnum\time>720~pm \else~am\fi}

\newtheorem{theorem}{Theorem}[section]

\newtheorem{corollary}[theorem]{Corollary}
\newtheorem{proposition}[theorem]{Proposition}
\newtheorem{remark0}[theorem]{Remark}
\newtheorem{example0}[theorem]{Example}
\newtheorem{definition}[theorem]{Definition}

\newenvironment{example}{\begin{example0}\rm}{\end{example0}}
\newenvironment{remark}{\begin{remark0}\rm}{\end{remark0}}

\newcommand{\propref}[1]{Proposition~\ref{#1}}
\newcommand{\thmref}[1]{Theorem~\ref{#1}}

\newcommand{\corref}[1]{Corollary~\ref{#1}}
\newcommand{\exref}[1]{Example~\ref{#1}}
\newcommand{\secref}[1]{Section~\ref{#1}}

\def\max{{\mathfrak{m}}}                   

\def\res{{\mathbf{k}}}

\def\supp{{\rm{Supp}}}

\def\HF{{\operatorname{H\!F}}}
\def\HP{\operatorname{H\!P}}

\def\deg{\operatorname{deg}}

\def\dim{\operatorname{dim}}
\def\ann{\operatorname{Ann}}

\def\ker{\operatorname{Ker}}
\def\length{\operatorname{Length}}

\def\spec{\operatorname{Spec}}

\def\inf{\operatorname{Inf}}

\def\OX{{\mathcal O}_{X}}
\def\OXS{{\mathcal O}_{\widetilde{X}}}
\def\OXP{{\mathcal O}_{X'}}

\def\OX0{{\mathcal O}_{X}}

\def\OCn0{{\mathcal O}_{(\res^n,0)}}
\def\Cn0{({\res}^n,0)}

\begin{document}
\title[How to determine a curve singularity]{{\bf
How to determine a curve singularity}}
\author[J. Elias]{J. Elias ${}^{*}$}
\thanks{${}^{*}$
Partially supported by PID2019-104844GB-I00\\
\rm \indent 2020 MSC:  Primary
13H10; Secondary 14B05; 13H15}
\address{Joan Elias
\newline \indent Departament de Matem\`{a}tiques i Inform\`{a}tica
\newline \indent Universitat de Barcelona (UB)
\newline \indent Gran Via 585, 08007
Barcelona, Spain}  \email{{\tt elias@ub.edu}}

\date{\today}

\begin{abstract}
We characterize the finite codimension sub-$\res$-algebras of $\res[\![t]\!]$ as the solutions of a computable finite family of higher differential operators.
For this end, we establish a duality between such a sub-algebras and the finite codimension
$\res$-vector spaces of $\res[u]$, this ring acts
on $\res[\![t]\!]$ by differentiation.
\end{abstract}

\maketitle

\section{Introduction }

It is well know that the normalization of a curve $X$ is a non-singular curve $Y$.
Serre considers in \cite[Chap IV]{Ser59} the opposite direction, he showed how to construct a curve $X$ from a given non singular curve $Y$ such that this curve is the normalization of $X$.
This idea appears in several different contexts.
For instance,  in \cite{Gor69}, \cite{New74}, \cite{GLTU22}, and the references therein,
is studied how to determine the finite codimension sub-$\res$-algebras $B$ of $\res[t]$.
Notice that, in this case,  $X=\spec(B)$ is an algebraic curve and the affine line $Y=\spec(\res[t])$  is its normalization.
These sub-algebras are defined recursively on the codimension by linear and higher differential conditions.
Only for low codimensions  explicit conditions are known.
Since not all higher differential conditions define sub-algebras of $\res[t]$,  it is an open problem the  characterization of  families of linear higher differential operators defining finite codimension sub-$\res$-algebras of $\res[t]$, see \cite{GLTU22}.

In the search of one-dimensional reduced local rings with locally decreasing Hilbert function, Roberts constructed such a local rings as connex, finite codimension sub-$\res$-algebras of $\prod_{i=1}^r \res[t_i]$ defined by linear and  first order differentials conditions,   \cite{GR83}.
See \cite{Eli93a} for the proof of Sally's conjecture on the monotony of Hilbert functions of one-dimensional Cohen-Macaulay local rings.

In this paper we consider the local complete case.
We characterize the finite codimension sub-$\res$-algebras $B$ of $\Gamma=\res[\![t]\!]$ as the solutions of a computable finite codimension $\res$-vector space $B^{\perp}\subset \Delta=\res[u]$ of higher differential operators, \thmref{matlislike}.
For this purpose, we establish a Macaulay-like duality
between finite codimension sub-$\res$-algebras $B$ of
$\Gamma$ and finite codimension $\res$-vector subspaces $B^{\perp}$, so-called algebra-forming vector spaces, of the polynomial ring $\Delta$.
The polynomial ring $\Delta$ acts on $\Gamma$ by
differentiation as in Macaulay's duality, see \cite{IK99}, \cite{Eli18}, \cite{ER17}, \cite{ER21}.
At the end of \secref{secMac} we describe the linear maps   $B_2^{\perp}\rightarrow B_1^{\perp}$ induced by  $\res$-algebra morphisms $B_1\rightarrow B_2$ between two finite codimension $\res$-algebras $B_1$, $B_2$.

In section 4 we study the algebra-forming vector spaces, showing that such a condition can be checked effectively, \propref{AF2}.
After this we prove that  for any finite codimension $\delta$ $\res$-algebra $B$ there exist a finite filtration of $\res$-algebras, so-called standard filtration of $B$,
$B=B_0\subset B_1\subset \cdots \subset B_{\delta}=\Gamma$ such that $\dim_{\res}(B_{i+1}/B_i)=1$
for $i=0,\dots,\delta-1$.
As corollary of this construction we get that we only need to consider algebra-forming single elements in order to define recursively a finite codimension $\res$-algebras.
Moreover, we show how to recover the standard filtration by considering recursively derivations of the local rings appearing in the filtration, \corref{derivations}.

Section 5 is devoted to study the inverse system of monomial $\res$-algebras and the special case of monomial Gorenstein algebras.
We end the section relating the inverse system of a curve singularity with its generic plane projection and its saturation.

In the last section we link $B^{\perp}$ with the canonical module of $B$, \propref{canon}.

\medskip
The computations of this paper are performed  by using the computer algebra system  Singular,  \cite{DGPS}.

\section{Preliminaries}

Let  $R$  denote the power series ring $ \res[\![x_1, \dots, x_n]\!] $
over an algebraically closed characteristic zero  field $\res$ and
we denote by $\max=(x_1,\cdots, x_n)$ its maximal  ideal.

Let $A$ be  a one-dimensional local ring with maximal ideal $\max$.
We denote by $\HF_A$ the Hilbert function of $A$, i.e.
$\HF_A(i)=\length_A(\max^i/\max^{i+1})$, $i\ge 0$.
It is well known that $\HF^0_A(i)=e_0(A)$, $i\gg 0$, where $e_0(A)$ is the multiplicity of $A$.
The first integral of $\HF_A$ is defined by, $i\ge 0$,
$$\HF^1_A(i)=\sum_{j=0}^i \HF_A(j)= \length_A(A/\max^{i+1}).$$
We write $\HF_A^0=\HF_A$.
There exists an integer $e_1(A)$ such that
$\HF^1_A(i)=e_0(A)( i+1) -e_1(A)$ for $i\gg 0$;
the (first) Hilbert polynomial  is $\HP^1_A(T)=e_0(A)( T+1) -e_1(A)$.
See \cite[Chapter XII]{Mat77} for the basic properties of the Hilbert functions of one-dimensional
Cohen-Macaulay local rings.

A branch $X$ is an irreducible curve singularity of $(\res^n,0)=\spec(R)$, i.e. $X$ is a one-dimensional, integral  scheme $X=\spec(R/I)$; we write $\OX0=R/I$ and $I(X)=I$.

Let $\nu: \overline{X}=\spec (\overline{\OX0})\longrightarrow (X,0)$ be the normalization of $(X,0)$, where
$\overline{\OX0}\cong\res[\![t]\!]$ is the integral closure of $\OX0$ on its full field of fractions ${\mathrm{ tot}}(\OX0)$.
The singularity order of $X$ is
$
\delta(X)=\dim_{\res}\left({\mathcal O}_{\overline{X}}/\OX0\right).
$
We denote by $\mathcal C$ the conductor of the finite extension
$\nu^*: \OX0 \hookrightarrow
\overline{\OX0}$
and by $c(X)$ the dimension of $\overline{\OX0}/\mathcal C$.

Given a set of non-negative integers $1\le a_1<\cdots <a_n$
we consider the monomial curve singularity $X(a_1,\cdots, a_n)$ defined by the parameterization
$$
\begin{array}{cccc}
    \gamma:&R &  \longrightarrow &\res[\![t]\!]\\
    &x_i & \mapsto & t^{a_i}
\end{array}
$$
i.e. $I(X(a_1,\cdots, a_n))=\ker(\gamma)$.
If  $gcd(a_1,\cdots, a_n)=1$ then  the induced map
$$\gamma:R/I(X(a_1,\cdots, a_n))\longrightarrow \res[\![t]\!]$$
is the normalization map of
$\mathcal O_{X(a_1,\cdots, a_n)}=R/I(X(a_1,\cdots, a_n))=\res[\![t^{a_1},\dots, t^{a_n}]\!]$.

We denote by $D_X$ the semigroup of values of $X$: the set of integers  $v_t(f)=ord_t(t)$ where $f\in  \OX0\setminus \{0\}$.
It is easy to see that $\delta(X)=\#(\mathbb N\setminus D_X)$.
If $B$ is a finite codimension sub-$\res$-algebra of $\Gamma$ then $X=\spec(B)$ is branch.
We write $D_B=D_X$.

Let  $\omega_{X}$ be the dualizing module of $X$;
we can consider the composition of $\OX0$-module morphisms
$$
\gamma_X: \Omega_{X} \longrightarrow
\nu_* \Omega_{\overline{X}}\cong
\nu_* \omega_{\overline{X}} \longrightarrow
 \omega_{X}.
$$
Let $d: \OX0 \longrightarrow  \Omega_{X}$ the universal derivation, then we have a $\res$-linear map $\gamma_X  d$ that we also denote
by $d: \OX0 \longrightarrow  \omega_{X}.$
Recall that the Milnor number of $X$ is  $\mu(X)=\dim_{\res}(\omega_{X}/d \OX0)$, \cite{BG80}.
Since we only consider branches we have that $\mu(X)=2\delta(X)$,  \cite[Proposition 1.2.1]{BG80}.
Notice that $X$ is non-singular iff $\mu(X)=0$ iff $\delta(X)=0$ iff $c(X)=0$.

We denote by $\pi:Bl(X)\longrightarrow X$ the blowing-up of $X$ on its closed point.
The fiber of the closed point of $X$ has a finite number of closed points: the so-called points of the first neighborhood of $X$.
We can iterate the process of blowing-up until we get the normalization of $X$, see \cite{Nor55} and \cite{Cut-ResSing}.
We denote by $\inf(X)$ the set of infinitely near points of $X$.
The curve singularity defined by an infinitely point $p$ of $X$ will be denote by $(X,p)$;
we set $(X,0)=X$.

\medskip
\begin{proposition}
\label{basic}
Let $X$ be a branch.
Then

\noindent
$(i)$ $$
\delta(X)=\sum_{p \in \inf(X)} e_i(X,p)
$$

\noindent
$(ii)$
It holds
$$
 e_0(X)-1\le e_1(X) \le \delta(X)\le \mu(X)
$$
and $e_1(X)\le \binom{e_0(X)}{2} -  \binom{n-1}{2}$.

\noindent
$(iii)$ If $X$ is singular then $\delta(X)+1 \le c(X)\le 2 \delta(X)$, and $c(X) =  2 \delta(X)$
if and only if $\OX0$ is a Gorenstein ring.
\end{proposition}
\begin{proof}
$(i)$ \cite{Nor59a}.
$(ii)$ \cite[Proposition 1.2.4 (i)]{BG80}, \cite{Nor59a},  \cite{Eli90}, \cite{Eli01}.
$(iii)$ \cite[Proposition 7, pag. 80]{Ser59}, and \cite{BC77}.
\end{proof}

\medskip
\section{Macaulay-like duality}
\label{secMac}

In this section we establish a Macaulay-like duality for the family of sub-$\res$-algebras $B$ of $\Gamma=\res[\![t]\!]$ of finite codimension.
For the classical Macaulay's duality see \cite{IK99}, \cite{Eli18}, and for the generalization to higher dimension of Macaulay's duality see  \cite{ER17}, \cite{ER21}.
Recall that Macaulay's duality is a particular case of Matlis' duality, \cite{BH97}.

We write
$\Delta= \res[u]$;
$\Gamma$ is a $\Delta$-module with $\Delta$ acting  on $\Gamma$  by derivation.
This action denoted by $\circ$ is defined by
$$
\begin{array}{ clc}
 \circ: \Delta \times  \Gamma  &\longrightarrow &  \Gamma   \\
                            (g,f) & \to  & g \circ f = g (\partial_{t})(f)
                            \end{array}
$$

\noindent
where $  \partial_{t} $ denotes the derivative with respect to $t$.
This action induces a non-singular $\res$-bilinear perfect pairing:
\begin{equation}
\label{pairing}
\begin{array}{ lclc}
\perp: & \Delta \times \Gamma
  &\longrightarrow &  \res    \\
    & (g,f)&\mapsto & g\perp f=(g\circ f)(0)
\end{array}
\end{equation}

\medskip
\begin{definition}
Given a sub-$\res$-algebra $B$ of $\Gamma=\res[\![t]\!]$ we define
$B^{\perp}$ as the set of $g\in\Delta$ such that
$g\perp f=0$ for all $f\in B$.
Notice that $B^{\perp}$  is a $\res$-vector subspace of $\Delta$, this is, following the classic Macaulay's duality terminology, the inverse system of $B$.
Given a $\res$-vector subspace $V\subset \Delta$ we consider $\ann(V)\subset \Gamma$ as the set
of power series $f\in \Gamma$ such that $g\perp f=0$ for all
$g\in V$.
\end{definition}

Let $B$ be a finite codimension sub-$\res$-algebra  of $\Gamma$.
Then we have a non-singular $\res$-bilinear perfect pairing:
\begin{equation}
\label{pfp}
\begin{array}{ cclc}
 \perp: &B^{\perp} \times \frac{\Gamma}{B}&\longrightarrow &  \res   \\
        &             (g,\overline{f})&\mapsto   &  g\perp f
\end{array}
\end{equation}

\medskip
\noindent
We denote by $Perp(B)$ the $\res$-vector space of maps
$$
\begin{array}{cclc}
 g^{\perp}: &B &\longrightarrow &  \res   \\
        &f &\mapsto   &  g\perp f
\end{array}
$$
for all $g\in \Delta$.
These maps are the elements of the dual space of $B$ with finite support:
$g^{\perp}(\max_B^d)=0$ for $d>\deg(g)$.
We denote by $Der_{\res}(B)$  the $\res$-vector space of $\res$-derivations of $B$.
Since
$Der_{\res}(B)\cong (\max_B/\max_B^2)^*$, we can identify $Der_{\res}(B)$  with the $\res$-vector space of  elements $\sigma$
of the dual space of $B$ such that $\sigma(\max_B^2)=0$.

We have $Der_{\res}(B)\subset Perp(B)$,
this inclusion is strict.
Let us consider the codimension $8$ algebra  $B=\res[\![t^4,t^7,t^{17}]\!]$.
The linear map $(u^{11})^{\perp}:B\longrightarrow \res$ is not a derivation since $t^{11}\in \max_B^2$ and
$(u^{11}){\perp}(t^{11})=11!\neq 0$.

\bigskip
Next step is to characterize the vector $\res$-vector subspaces $B^{\perp}$ of $\Delta$ where $B$ ranges the family of  finite codimension sub-$\res$-algebras  of $\Gamma$.
First, we give some properties of $B^{\perp}$ that we will use along the paper.

Given a polynomial $g=\sum_{i=0}^d a_i u^i\in \Delta$ we denote by $\supp(g)$ the support of $g$: the finite set of integers $i$
such that $a_i\neq 0$.

\medskip
\begin{proposition}
\label{charV}
Let $B\subset \Gamma$ be a codimension $\delta$ sub-$\res$-algebra $B$ of $\Gamma$ and let  $\mathcal C=(t^c)$ the conductor of the extension $B\subset \Gamma$.
Then
\begin{enumerate}
\item[(1)]
$\dim_{\res}(B^{\perp})=\delta$.
\item[(2)]
For all $g\in B^{\perp}$ we have
$ \supp(g)\subset [1, c-1]$, and
$$
u^{[1, e_0(B)-1]}=\{u^i; i\in [1, e_0(B)-1]\} \subset B^{\perp}\subset \langle u, u^2, \dots , u^{c-1}\rangle.
$$
\item[(3)]
The following conditions are equivalent:
\begin{enumerate}
    \item[(i)] $\delta=0$,
    \item[(ii)] $B=\Gamma$,
    \item[(iii)] $B^{\perp}=0$,
    \item[(iv)] $B^{\perp}\subset \langle u^2,u^3,\dots \rangle$.
\end{enumerate}
\end{enumerate}
\end{proposition}
\begin{proof}
$(1)$
Since $\perp$ is a $\res$-bilinear perfect pairing we get $\dim_{\res}(B^{\perp})=\delta$, see the equation (\ref{pfp}).

\noindent
$(2)$
Since $B$ is a $\res$-algebra we have $1\in B$, so
 if $g=\sum_{j\ge 0} a_iu^i\in B^{\perp}$  then $0=g\perp 1= a_0$.
 Hence $B^{\perp}\subset \langle u, u^2, \dots \rangle $.
We know that $(t^c)\subset B$ so for all $g=\sum_{j\ge 0} a_iu^i\in B^{\perp}$ we have
$$
0=g\perp t^{c+i}=(c+i)! a_{c+i}
$$
$i\ge 0$.
Hence, if $g\in B^{\perp}$ then $\deg(g)\le c-1$.
From this we deduce that $B^{\perp}\subset \langle u, u^2, \dots , u^{c-1}\rangle $.

Notice that $v_t(f)\ge e_0(B)$ for all $f\in B\setminus \{1\}$, so given $i\in [1, e_0(B)-1]$    we have $u^i\perp f=0$.
Hence $u^i\in B^{\perp}$ and then
$u^{[1, e_0(B)-1]}\subset B^{\perp}$.

\noindent
$(3)$
The condition of $(i)$ is equivalent to $(ii)$.
$(ii)$ trivially implies $(iii)$ and this implies $(iv)$.
If $B^{\perp}\subset \langle u^2,u^3,\dots \rangle$ then $t\in B$, since $B$ is a $\res$-algebra we get $(ii)$.
\end{proof}

\medskip
For all power series $f=\sum_{i\ge 0} b_i t^i\in \Gamma$ and given a non-negative integer $s\in \mathbb N$ we denote
by $[f]_{\le s}$ the truncated polynomial
$[f]_{\le s}=\sum_{i\ge 0}^s b_i t^i$.

Let $B$ be a finite codimension sub-$\res$-algebra of $\Gamma$ with conductor $c$.
Then $B$ is a finitely generated $\res$-algebra;
let $f_1,\dots,f_r$ be a system of generators of $B$ as $\res$-algebra.
We denote by $\natural_{B,d}$, $d\ge c-1$, the finite set of polynomials
$[f_1^{l_1}\cdots f_r^{l_r}]_{\le d}$ with
$l_i\ge 0$, $i=1,\dots,r$, and $l_1+\cdots+l_r\le d$.
We denote by $W(\{f_1,\dots, f_r\},d)\subset \Delta$ the
$\res$-vector space  generated by the polynomials
of $\natural_{B,d}$.
Notice that
$W(\{f_1,\dots, f_r\},d)+\langle t^{d+1}\rangle=W(\{f_1,\dots, f_r\},d+1)$.

\medskip
\begin{proposition}
\label{finite}
Let $B$ be a finite codimension sub-$\res$-algebra of $\Gamma$ with conductor $c$.
Then $B^{\perp}$ is the set of $g\in \Delta$ of degree at most $c-1$ and such that
$g\perp h=0$ for all $h\in \natural_{B,c-1}$.
\end{proposition}
\begin{proof}
Let  $f_1,\dots,f_r$ be a system of generators of $B$ as $\res$-algebra and let $\natural_{B,c-1}$ be the associated set of polynomials.

If $g\in B^{\perp}$ then $\deg(g)\le c-1$, \propref{charV}(2), so
$$
0=g\perp (f_1^{l_1}\cdots f_r^{l_r})=
g\perp [f_1^{l_1}\cdots f_r^{l_r}]_{\le c-1}.
$$
Hence $g\perp h=0$ for all $h\in \natural_{B,c-1}$.

Let $g \in \Delta$ be a polynomial with $\deg(g)\le c-1$ and such that
$g\perp h=0$ for all $h\in \natural_{B,c-1}$.
Any $f\in B$ can be written as
$$
f=\sum_{l_1,\dots,l_r\in \mathbb N} c_{l_1,\dots,l_r}f_1^{l_1}\cdots f_r^{l_r}
$$
with $c_{l_1,\dots,l_r}\in \res$.
Since $\deg(g)\le c-1$ we have
$$
g\perp f=
\sum_{l_1,\dots,l_r\in \mathbb N} c_{l_1,\dots,l_r}
(g\perp f_1^{l_1}\cdots f_r^{l_r})=
\sum_{l_1,\dots,l_r\in \mathbb N} c_{l_1,\dots,l_r}
(g\perp [f_1^{l_1}\cdots f_r^{l_r}]_{\le c-1})=0,
$$
so $g\in B^{\perp}$.
\end{proof}

\medskip
\begin{remark}
Notice that \propref{finite} shows that the computation of $B^{\perp}$ is effective.
In fact, in the set $\natural_{B,c-1}$ there are involved a finite number of monomials  and we only have to consider polynomials $g$ of degree at most $c-1$.
\end{remark}

\begin{remark}
Although $B^{\perp}$  is a $\res$-vector subspace of $\Delta$ for any sub-$\res$-algebra $B$ of $\Gamma$, not all $\ann(V)$ is a $\res$-algebra for a given $\res$-vector subspace $V\subset \Delta$.
In fact, let us consider the $\res$-vector subspace $V\subset \Delta$ generated by $u^2$.
Then $\ann(V)$ is the set of
$f=\sum_{i\ge 0}a_it^i\in \Gamma$ such that $a_2=0$.
This is not a $\res$-algebra because
$u^2\perp t=0$, so $t\in \ann(V)$ and $u^2\perp t^2=2\neq 0$, so $t^2\notin \ann(V)$.
\end{remark}

\medskip
\begin{definition}
A finite dimensional $\res$-vector subspace $V\subset \Delta$ is so-called algebra-forming with respect to a $\res$-algebra $B\subset \Gamma$ iff the following conditions hold:
\begin{enumerate}
\item[(a)]
$g(0)=0$ for all $g\in V$ and,
\item[(b)]
for all $f\in B$ such that $g\perp f=0$ for all $g\in V$ it holds $g\perp f^2=0$ for all $g\in V$.
\end{enumerate}
An element $g\in \Delta$ is so-called algebra-forming with respect to $B$ if $V=\langle g\rangle$ is algebra-forming with respect to $B$.
\end{definition}


\begin{example}
\label{toy-example}
Let us consider the codimension $\delta=4$ algebra
$B=\res[\![t^3+t^4,t^5]\!]$ of $\Gamma$.
The conductor of $B$ is $c=8$.
Then $B^{\perp}$ is the set of polynomials
$g\in \Delta$ of degree at most $7$ such that
$g\perp f=0$ for $f\in \natural_{B,c-1}=\{t^3+t^4, t^5, t^6+2t^7\}$.
A simple computation shows that
$B^{\perp}$ is the $\res$-vector space generated by the $4$ linear independent polynomials
$u, u^2, u^3-\frac{1}{4}u^4, u^6-\frac{1}{2.7}u^7$.
Let us consider
$$
B_2=\res[\![t^3,t^4,t^5]\!]\subset
B_3=\res[\![t^2,t^3]\!],
$$
then we have
$B_2=\ann\langle u^2\rangle\cap B_3$, i.e. $u^2$ is an algebra-forming element with
respect to $B_2$.
\end{example}

\medskip
In the following result we prove that, in fact, if $V\subset \Delta$ is  algebra-forming with respect to a $\res$ algebra $B\subset \Gamma$
 then $\ann(V)\cap B$ is a sub-$\res$-algebra of $\Gamma$.

\medskip
\begin{proposition}
\label{AF}
Let $V\subset \Delta$ be an algebra-forming $\res$-vector subspace with respect to a $\res$-algebra $B\subset \Gamma$.
Then $\ann(V)\cap B$ is a sub-$\res$-algebra of $\Gamma$.
\end{proposition}
\begin{proof}
Clearly $C=\ann(V)\cap B$ is  a $\res$-vector subspace of $\Gamma$.
Given  $f_1, f_2\in C$  we have that $f_1+f_2\in C$ and from
$$
f_1 f_2=\frac{1}{2}((f_1+f_2)^2-f_1^2-f_2^2)
$$
we deduce that $g\perp(f_1 f_2)=0$, i.e.
$f_1f_2\in C$.
Since $g(0)=0$ for all $g\in V$ we get $1\in C$, so $C$ is a  sub-$\res$-algebra of $\Gamma$.
\end{proof}

\medskip
The following result is an extension of Macaulay's duality to finite codimension sub-$\res$-algebras $B\subset \Gamma$.

\medskip
\begin{theorem}
\label{matlislike}
Given a non-negative integers $\delta> 0$ and $c\ge \delta + 1$,
there is a one-to-one  correspondence $\perp$ between the following sets:
\begin{enumerate}
\item   sub-$\res$-algebras $B$ of $\Gamma$ of codimension $\delta$ as $\res$-vector spaces such that the conductor of $B\subset \Gamma$ is $(t^c)$,
\item
algebra forming, with respect to $\Gamma$, $\res$-vector subspace $V\subset \Delta$ of dimension $\delta$, generated by polynomials of degree at most $c-1$ and such that there is a polynomial $g\in V$ with $\deg(g)=c-1$.
\end{enumerate}
\noindent
This correspondence is inclusion reversing: given two  sub-$\res$-algebras $B_1$ and $B_2$ of $\Gamma$,
$B_1\subset B_2$ if and only if $B_2^{\perp}\subset B_1^{\perp}$.
\end{theorem}
\begin{proof}
Let $B$ be a sub-$\res$-algebra $B$ of $\Gamma$.
Since we have a non-singular $\res$-bilinear pairing:
$$
\begin{array}{ cclc}
 \perp: &B^{\perp} \times \frac{\Gamma}{B}&\longrightarrow &  \res   \\
        &             (g,\overline{f})&\mapsto   &  g\perp f
\end{array}
$$
we get that $B^{\perp}$ is a $\res$-vector subspace of dimension $\delta$ of $\Delta$.
By definition $B^{\perp}$ is  algebra-forming with respect to $\Gamma$.
Being $c$ the conductor we have $(t^c)\subset B$, so
$\deg(g)\le c-1$ for all $g\in B^{\perp}$ and there exist $g\in B^{\perp}$ of degree $c-1$.

Let $V$ be an algebra forming, with respect to $\Gamma$, $\res$-vector subspace satisfying the conditions of (2).
Let us consider the $\res$-algebra $B=\ann(V)$.
From the perfect pairing (\ref{pairing})
we get that the codimension of $B$ in $\Gamma$ is $\delta$.
Since $V$ is generated by polynomials of degree at most $c-1$ we have that
$(t^c)\subset B$, so the conductor of $B$ is at most $c$.
Furthermore, since there is $g\in V$ with $\deg(g)=c-1$ we deduce that $c$ is the conductor of $B$.

It is straightforward to prove the inclusion reversing from the definition of the inverse system $B^{\perp}$.
\end{proof}

\medskip
We end this section by describing  the $\res$-linear maps   $B_2^{\perp}\longrightarrow B_1^{\perp}$ induced by  $\res$-algebra isomorphisms $B_1\longrightarrow B_2$ between two finite codimension $\res$-algebras $B_1$ and $B_2$ of $\Gamma$.
Let $c$ be an integer bigger than the conductors of $B_1$ and $B_2$.

The perfect pairing $(\ref{pairing})$ induce  a perfect pairing
$$
\begin{array}{ lclc}
\perp: & \Delta_{\le c-1} \times \frac{\Gamma}{(t^c)}
  &\longrightarrow &  \res    \\
    & (g,\overline{f})&\mapsto & g\perp f=(g\circ f)(0)
\end{array}
$$
where $\Delta_{\le c-1}$ is the $\res$-vector space of polynomials of degree at most $c-1$.
We consider the usual $\res$-vector basis of $\Gamma/(t^c)$ of the cosets of $t^i$, $i=0,\dots, c-1$.
Its dual basis is $\frac{1}{i!} u^i$, $i=0,\dots, c-1$, since
$$
\left(\frac{1}{i!} u^i\right)\perp t^j=\delta_{i,j}
$$
$1\le i, j \le c-1$.

The $\res$-algebra $B_i$ has conductor at most $c$ so we can consider that $B_i\subset \Gamma/(t^{c})$, $i=1,2$.
On the other hand, from \propref{charV} we have that $B_i^{\perp}\subset \Delta_{\le c-1}$, $i=1,2$.

If $B_1$ is isomorphic to $B_2$ by $\phi$ then
their normalizations are isomorphic:
$$\Gamma =\overline{B_1}\overset{\overline{\phi}}{\cong}\overline{B_2}=\Gamma.
$$
This automorphism is determined by a power series $h(t)\in (t)$ such
that $u\perp h\neq 0$ and
$$
\begin{array}{ lclc}
\overline{\phi}: & \Gamma &\longrightarrow &  \Gamma    \\
    & f &\mapsto & f(h).
\end{array}
$$
Then we have an isomorphism  of $\res$-vector spaces
$$
\frac{\Gamma}{B_1}\overset{\overline{\phi}}{\longrightarrow} \frac{\Gamma}{B_2}
$$
and the perfect pairing induces a $\res$-vector isomorphism
$$
\phi^*:B_2^{\perp}\longrightarrow B_1^{\perp}.
$$

The matrix $M_{\phi}$ associated to $\phi$ in the basis
${t^i}$,  $i=0,\dots, c-1$, is the $c\times c$ matrix whose columns are the coefficients
of $\phi({t}^i)={h}^i$,  $i=0,\dots, c-1$,  with respect to this basis.
Hence the matrix of $\phi^*:B_2^*=B_2^{\perp}\longrightarrow B_1^*=B_1^{\perp}$
with respect to the basis $\frac{1}{i!} u^i$, $i=0,\dots, c-1$, is the transpose matrix $^\tau M_{\phi}$ of $M_{\phi}$.

\medskip
\begin{example}
Let $B_2\subset \Gamma$ be a $\res$-algebra generated by
two elements $f_1$, $f_2$ with
$v_t(f_1)=2$ and $v_t(f_2)=7$.
We may assume that $f_1=t^2+$monomials of higher degree.
Then $B_2$ is of finite codimension $\delta=3$ and conductor $c=6$.

Since $\Gamma$ is complete there exist a power series $h\in (t)$ such that $h^2=f_1$;
we write $h=t+h_2 t^2+\dots +h_5 t^5+\dots$.
Notice that $\Gamma=\res[\![h]\!]$.

Let ${\phi}$ the automorphism of $\Gamma$ defined by $h$, i.e. ${\phi}(f)=f(h)$.
Then $\phi^{-1}(B_2)$ is a $\res$-algebra $B_1$ generated by $f_1'=t^2$ and $f_2'(h)$ such that $v_h(f_2'
)=7$.
After a change of generators $B_1$ is generated by $f_1'=t^2$ and $f_2'=t^7$.

The induced isomorphism $\phi:B_1\longrightarrow B_2$  has the following $6 \times 6$ associated matrix with respect the basis
${t^i}$,  $i=0,\dots, 5$,
$$
M_{\phi}=\left(
  \begin{array}{cccccc}
    1 & 0 & 0 & 0 & 0 & 0 \\
    0 & 1 & 0 & 0 & 0 & 0 \\
    0 & h_2 & 1 & 0 & 0 & 0 \\
    0 & h_3 & 2h_2 & 1 & 0 & 0 \\
    0 & h_4 & 2h_3+h_2^2 & 3h_2 & 1 & 0 \\
    0 & h_5 & 2b_4+2h_2h_3 & 3h_3+3h_2^2 & 4h_2 & 1 \\
  \end{array}
\right)
$$
Then the matrix of the isomorphism $\phi^*:B_2^{\perp}\longrightarrow B_1^{\perp}$
with respect to $\frac{1}{i!} u^i$, $i=0,\dots, 5$, is $M_{\phi}^\tau$.
Since $B_1$ is the monomial $\res$-algebra
$\res[\![t^2,t^7]\!]$, the $\res$-vector space $B_1^{\perp}$ is generated by $u, u^3, u^5$.
From this we can compute $B_2^{\perp}$ by considering
$(^\tau M_{\phi})^{-1}$.
\end{example}

\medskip
\section{Algebra-forming vector spaces}

\medskip
The first goal of this section is to characterize the algebra-forming $\res$-vector spaces.

\medskip
\begin{proposition}
\label{AF2}
Let $B$ be a $\res$-sub-algebra of finite codimension  of $\Gamma$
with conductor $c$, and let $f_1,\dots, f_s$ be a system of generators of $B$.
Given an integer $d\ge c-1$,
let $h_1,\dots, h_m$ be a system of generators of $W(\{f_1,\dots, f_s\},d)$.

Let $V$ be a  dimension $\delta$ $\res$-vector subspace of $(u)\subset\Delta$ generated by polynomials of degree at most $d-1$.
Let $g_1,\dots,g_{\delta}\in V$ be a basis of $V$.

Then $V$ is algebra-forming with respect to $B$ iff
for all $r$-upla $(\lambda_1,\dots,\lambda_m)\in \res^m$
such that
\begin{equation}
\label{Lin}
 \sum_{j=1}^m \lambda_j (g_i\perp h_j)=0
\end{equation}
for all $i=1,\dots,\delta$, then
\begin{equation}
\label{Quad}
\sum_{j=1}^m \lambda_j^2 (g_i\perp h_j^2)+
2 \sum_{j=1,l=1,j\neq l}^m \lambda_j \lambda_j (g_i\perp h_j h_l)=0
\end{equation}
for all $i=1,\dots,\delta$.
\end{proposition}
\begin{proof}
From \propref{charV} we have to prove that for all $f\in B$ such that $g\perp f=0$ for all $g\in V$ we have that
$g\perp f^2=0$ for all $g\in V$.
Since the polynomials of $V$ are of degree at most $d-1$
we only have to prove that for all
$f\in W=W(\{f_1,\dots, f_s\},d)$ such that $g\perp f=0$ for all $g\in V$, we have that
$g\perp f^2=0$ for all $g\in V$.

A general element of $W$ can be written as $f=\sum_{j=1}^m \lambda_j h_j$.
Hence the condition $g_i\perp f=0$ is equivalent to
$$
\sum_{j=1}^m \lambda_j (g_i\perp h_j)=0
$$
for all $i=1,\dots,\delta$.
Similarly, the condition
$g_i\perp f^2=0$ is equivalent to
$$
\sum_{j=1}^m \lambda_j^2 (g_i\perp h_j^2)+
2 \sum_{j=1,l=1,j\neq l}^m \lambda_j \lambda_j (g_i\perp h_j h_l)=0
$$
for all $i=1,\dots,\delta$.
\end{proof}

\medskip
\begin{remark}
\label{def-alg-for}
The set of points $(\lambda_1,\dots,\lambda_m)\in \mathbb P^{m-1}_{\res}$ satisfying the identities of (\ref{Lin}) form a
linear subvariety $L$, and the points satisfying the identities of (\ref{Quad}) defines a subvariety  $Q\subset \mathbb P^{m-1}_{\res}$ intersection of $\delta $ quadrics.
Hence $V$ is algebra forming with respect to $B$ iff $L\subset Q$.
This is a computable condition.
\end{remark}

\medskip
\begin{definition}
\label{standard-fil}
Let $B$ be a sub-$\res$-algebra of finite codimension $\delta$ of $\Gamma$ and conductor $c$.
Let $D$ be the semigroup of $B$;
we write the set $t^{\mathbb N\setminus D_B}=\{t^i; i\in \mathbb N\setminus D_B \}$ as
$g_1=t^{c-1},\dots, g_{\delta}=t$.
Then we define the so-called standard filtration of $B$ as follows:
$B_i$ is the $\res$-algebra generated by B and $g_1,\dots, g_i$ for $i=1, \dots, \delta$; we set $B_0=B$.
Notice that $B_{\delta}=\Gamma$ and that we have
$$
B=B_0\subset B_1\subset \dots \subset B_{\delta}=\Gamma
$$
and $\dim_{\res}(B_{i+1}/B_i)=1$, $i=0,\dots, \delta-1$.
\end{definition}

\medskip
After the definition of standard filtration  we only have to consider algebra-forming elements $g\in \Delta$, with respect a suitable sub-$\res$-algebras of $\Gamma$,
in order to define a $\res$-algebra recursively.
The algebra-forming elements are not unique as the following example shows.

\medskip
\begin{example}
\label{toy-example-3}
Let us consider the \exref{toy-example}.
The standard filtration of $B$ is:
$$
B=\res[\![t^3+t^4,t^5]\!]\subset
B_1=\res[\![t^3+t^4,t^5,t^7]\!]\subset
B_2=\res[\![t^3,t^4,t^5]\!]\subset
B_3=\res[\![t^2,t^3]\!]\subset \Gamma.
$$
The chain of $\res$-algebras is defined as follows.
The cosets of $t,t^2, t^4, t^7$ in $\Gamma/B$ form a basis of $\Gamma/B$ as $\res$-vector space.
Then
$B_1$ is the $\res$-algebra generated by $B$ and $t^7$,
$B_2$ is the $\res$-algebra generated by $B_1$ and $t^4$,
$B_3$ is the $\res$-algebra generated by $B$ and $t^2$, and finally $\Gamma$
is the $\res$-algebra generated by $B$ and $t$.

We know that $B^{\perp}$ is a $4$ dimensional $\res$-vector space generated by
$u, u^2, u^3-\frac{1}{4}u^4, u^6-\frac{1}{2.7}u^7$;
we have
$B_3=\ann\langle u\rangle$,
$B_2=\ann\langle u^2\rangle\cap B_3$,
$B_1=\ann\langle u^3-\frac{1}{4}u^4\rangle\cap B_2$,
$B=\ann\langle u^6-\frac{1}{2.7}u^7\rangle\cap B_1$.
On the other hand, the $\res$-algebra
$C_1=\res[\![t^3+t^5,t^4]\!]\subset B_1$ can be obtained as
$$
C_1=\ann\langle u^3-\frac{1}{4.5}u^5\rangle\cap B_2,
$$
i.e.
$u^3-\frac{1}{4.5}u^5$ is an algebra-forming element with
respect to $B_2$.
Notice that $B_1$ and $C_1$ are non analytically
isomorphic codimension one $\res$-algebras of $B_2$ .
\end{example}

\medskip
Next we show how to build the standard filtration by using derivations.

\medskip
\begin{proposition}
Let $C\subset B$ be two sub-$\res$-algebras of $\Gamma$ such that
$\dim_{\res}(B/C)=1$.
There exist $\alpha\in Der_{\res}(B)$ such that
$\ker(\alpha)=C$.
\end{proposition}
\begin{proof}
If we denote by $\max_B$ the maximal ideal of $B$ then
$\max_C\subset \max_B$,  $\dim_{\res}(\max_B/\max_C)=1$ and  $\max_B^2\subset \max_C$.
Since we have
$$
\frac{\max_C}{\max_B^2} \subset \frac{\max_B}{\max_B^2}
$$
we deduce that there exists a linear form $\alpha: \frac{\max_B}{\max_B^2}\longrightarrow \res$
such that
$\ker(\alpha)=\frac{\max_C}{\max_B^2}$.
From this we get the claim.
\end{proof}

\medskip
\begin{corollary}
\label{derivations}
Let $B$ be a sub-$\res$-algebra of finite codimension $\delta$ of $\Gamma$.
Let us consider the standard filtration of $B$:
$$
B=B_0\subset B_1\subset \dots \subset B_{\delta}=\Gamma.
$$
For all $i=1,\dots,\delta$ there exists a derivation
$\partial_{l_i}\in Der_{\res}(B_i)$, $l_i\in\max_{B_i}$, such that
$\ker(\partial_{l_i})=B_i$.
\end{corollary}

\medskip
\begin{example}
Let us consider the \exref{toy-example-3}.
The element $u^{\perp}$ corresponds to the derivation $\partial_t$ of $\Gamma$ defined by $t$, so
$B_3=\ker(\partial_t)$.
The maximal ideal of $B_3$ is minimally generated by $t^2, t^3$, the element $(u^2)^\perp$ is the derivation $\partial_{t^2}\in Der_{\res}(B_3)$,
so $B_2=\ker(\partial_{t^2})$.
The maximal ideal of $B_2$ is minimally generated by $t^3, t^4, t^5$.
The element $(u^3-\frac{1}{4}u^4)^{\perp}$ is the derivation
$\partial_{t^3-\frac{1}{4}t^4}\in Der_{\res}(B_2)$, so
$B_1=\ker(\partial_{t^3-\frac{1}{4}t^4})$.
Finally, $\partial_{t^7}\in Der_{\res}(B_1)$ and $B=\ker(\partial_{t^7})$.
\end{example}

\medskip
\section{Monomial algebras}

In this section we first compute  the inverse system of a monomial $\res$-algebra.
After this, we characterize monomial Gorenstein curve singularities in terms of its
inverse system.
We end the section relating the inverse system of a curve singularity with its generic plane
projection and its saturation.

The following result  it is easy to deduce from the proof of the second part of \propref{charV}(2).

\medskip
\begin{proposition}
\label{monomial}
Let $D$ be an additive  sub-semigroup of $\mathbb N$ with finite complement.
Then $B^{\perp}$ is the $\res$-vector space generated by:
$g_i=u^i$ for $i\in \mathbb N\setminus D$.
\end{proposition}

\medskip
\begin{example}
\label{delta1}
Let $B$ be a sub-$\res$-algebra of $\res[\![t]\!]$ of codimension $\delta =1$.
Then $B$ is the $\res$-algebra $B=\res[\![D]\!]$ where $D$ is the sub-semigroup of $\mathbb N$ generated by $2, 3$.
Hence $B^{\perp}$ is the $\res$-vector space generated by $u$, i.e. $B$ is the set of power series  $f=\sum_{i\ge 0} b_it^i \in \res[\![t]\!]$
with  $u\perp f= b_1=0$.
See  \cite{Ser59} Example b), Section 4 of Chapter IV, and
\cite[Section 22]{GLTU22}.
\end{example}

\begin{example}
\label{delta2}
Assume now that $B$ is sub-$\res$-algebra of $\res[\![t]\!]$ of codimension
$\delta= 2$.
Then its semi-group $D_B$ is $D_1=\langle 2, 5\rangle$ or  $D_2=\langle 3, 4\rangle$.
In the first case $B$ is generated as $\res$-algebra by $f_1=t^2+b_3 t^3$ and $f_2=t^5$.
 The conductor is $c=4$.
Then $B^{\perp}$ is generated by $g_1=u$,
$g_2=6 b_3 u^2 +u^3$.
In the second case $B$ is the monomial $\res$-algebra $B=\res[\![D_2]\!]$ so $B^{\perp}$ is the sub-$\res$-algebra generated by
$g_1=u$ and  $g_2=u^2$.
 The conductor is $c=5$.
See \cite[Section 23]{GLTU22}.
It is known that the algebras of the first case are all analytically isomorphic to
$\res[\![D_1]\!]$.
\end{example}

\medskip
The inverse system of a monomial Gorenstein $\res$-algebra case can be handled.
Let us recall the definition of symmetric semi-group and the celebrate result of Kunz.

\medskip
\begin{definition}
We say that a sub-semigroup $D$ of $\mathbb N$ such that $\#(\mathbb N \setminus D)<\infty$ and with conductor $c$ is symmetric if the condition  $t\in D$ is equivalent to $c-1-t\notin D$.
\end{definition}

\medskip
Kunz proved that the ring $\res[\![D]\!]$ is Gorenstein ring if and only if $D$ is a symmetric semigroup,\cite{Kun70}.
This symmetry is inherited by $B^{\perp}$.

\medskip
\begin{proposition}
\label{Gor}
Let $D$ be a sub-semigroup  of $\mathbb N$
such that $\#(\mathbb N \setminus D)<\infty$ and conductor $c$.
The following conditions are equivalent:
\begin{enumerate}
    \item $\res[\![D]\!]$ is Gorenstein,
    \item for all $g\in \res[\![D]\!]^{\perp}$ it holds $t^{c-1} g(1/t)\in \res[\![D]\!]$.
\end{enumerate}
\end{proposition}
\begin{proof}
Since $B=\res[\![D]\!]$ is a monomial $\res$-algebra we know that $B^{\perp}$ is generated by
$g=\sum_{i=1}^{c-1}a_i u^i$ such that $a_i=0$ for $i\in D$, \propref{monomial}.
Then the exponents of the non zero terms of $t^{c-1} g(1/t)$ are $c-1-i$ with $i\notin D$.
Then the claim is equivalent to the  symmetry of $D$, i .e. the Gorensteinness of $B$.
\end{proof}

\medskip
\begin{example}
Let $D$ be the semigroup generated by  $4, 6 ,$ and $ 9$.
This is a symmetric semigroup with conductor $c=12$.
The algebra $B=\res[\![D]\!]$ is Gorenstein and  isomorphic to $\res[\![x,y,z]\!]/I$ where
$I=(x^3-y^2, y^3-z^2)$.
Then $B^{\perp}$ is generated by the polynomials
$g=a_1 u+ a_2 u^2+ a_3 u^3 + a_5 u^5 + a_7 u^7$, $a_i\in\res$.
The polynomials $t^{11}g(1/t)= a_1 t^{10}+ a_2 t^9 + a_3 u^8 + a_4 u^6+a_5 u^4$ have all exponents in $D$.
The $\res$-vector space $B^{\perp}$
is generated by the following elements
$g_1=u, g_2=u^2, g_3=u^3, g_4=u^5, g_5=u^7$.
\end{example}

\medskip
Given a finite codimension subalgebra $B$ of $\Gamma$ we consider the curve singularity
$X=\spec(B)$ defined by $B$.
Let $X'$ be the generic plane projection of $X$, \cite{BGG80}, and let $\widetilde{X}$ be
the saturation of $X$, \cite{Zar-Sat-III} and the references therein.
We have
$$
\OXP\subset \OX0=B\subset \OXS\subset \Gamma
$$
and then
$$
\OXS^{\perp}\subset B^{\perp} \subset \OXP^{\perp}.
$$
We have, \cite{Eli88},
$$
\delta(\widetilde{X}) \le \delta(X) \le \delta(X')\le (e_0(X)-1) \delta(\widetilde{X})- \binom{e_0(X)-1}{2}
$$
From \cite[Proposition  1.6, pag. 971]{Zar65b} we know that $\widetilde{X}$ is also the saturation of $X'$.

On the other hand $\widetilde{X}$ is a monomial curve singularity.
Assume that the coset of $x_1$ in $B$ is $t^{e_0}$ with $e_0$ the multiplicity of $B$.
Since the rings are complete and the ground field is algebraically closed, we can  assumed it after a suitable election of the uniformization parameter of $\Gamma$.
Let $\{e_0; \beta_1,\dots, \beta_g\}$ be the characteristic of $X'$, \cite[Section 3, pag. 993]{Zar-Sat-III}, then
$\OXS$ is the monomial subalgebra with generators:
$$
\left\{
    \begin{array}{ll}
      t^{e_0}, & \hbox{} \\
      t^{s_\nu n_{\nu+1}\cdots n_g}, & \hbox{} m_{\nu}\le s_{\nu}\le [m_{\nu+1}/n_{\nu+1}], \nu=1,\dots, g-1 \\
      t^{m_g+i}, & \hbox{} 0\le i\le e_0-1
    \end{array}
  \right.
$$
where $\beta_{\nu}/e_0=m_{\nu}/n_1\dots n_{\nu}$ is the $\nu$-th characteristic exponent of $X'$, $\nu=1,\dots, g-1$, and $\gcd(m_i, n_i)=1$
for all $i=1,\dots,g$, \cite[Section 3, pag. 995]{Zar-Sat-III}.

The facts
$\OXS^{\perp}\subset B^{\perp}$ and \propref{delta1} can be useful in order to simplify the computation of $B^{\perp}$
as the next example shows.

\begin{example}
Let us consider the $\res$-algebra $B=\res[\![t^6, t^8+t^{11}, t^{10}+t^{13} ]\!]$; its saturation is
$\widetilde{B}=\res[\![t^6, t^8, t^{10}, t^{11}, t^{13}, t^{15}]\!]$, \cite[Example 2.5.1]{CC05}.
The sequence of multiplicities of the resolution of $X=\spec(B)$ is
$\{6,2,2,2,2,1,\dots\}$.
We can compute $\delta(X)$ by computing $e_1(C)$ where $C$ ranges the local rings of the resolution process, in this case we get $\{8,1,1,1,1,0,\dots\}$, so $\delta(X)=12$.
The semigroup of $B$ is
$D=\{0,6,8,10,12,14,16,18,19,20,22\rightarrow\}$, i.e. the conductor of $D$ is $22$.

On the other hand the semigroup of $\OXS$ is $\{0,6,8,10\longrightarrow\}$, its conductor is $10$.
Hence $\OXS^{\perp}$ is generated by $u^i$ with $i\in\{1,2,3,4,5,7,9\}$, and
$B^{\perp}$ is the set of polynomials
$g=\sum_{i=0}^{21} a_i u^i$ such that
$a_6=0$, $990 a_{11}- a_8=0$, $a_{12}=0$, $1716 a_{13}-a_{10}=0$, $a_{16}=0$, $4080 a_{17}-a_{14}=0$,
$a_{18}=a_{19}=a_{20}=a_{21}=0$.

\end{example}

\medskip
\section{The canonical module}

As in the Artin case we can relate the canonical module with the inverse system. In that case we have that if $I$ is an Artinian ideal then $I^{\perp}\cong E_{R/I}(\res)\cong \omega_{R/I}$, \cite{BH97}, \cite{Eli18}.
In the case of branches we can determine the "negative" part of the canonical module.

Let $X$ be a branch of $(\res^n,0)$ and $\overline{X}$ its normalization.
We  first describe the canonical module $\omega_X$ by using Rosenlicht's  regular differential forms,
\cite[IV 9]{Ser59},  \cite[Section 1]{BG80}, see also \cite{Eli16}.
We denote by $\Omega_{\overline{X}}(p)$ the set of meromorphic forms in $\overline{X}$
with a pole at most in  $p=\nu^{-1}(0)$.
Then Rosenlicht's differential forms are defined as follows:
$\omega^R_{X}$ is the set of $\nu_*(\alpha)$,  $\alpha \in \Omega_{\overline{X}}(p)$, such that
for all $F\in \OX0$
$$
{\rm res}_{p}(F \alpha)=0.
$$
Notice that we have a mapping that we also denote by
$$
d_R : \OX0 \longrightarrow
\Omega_{X} \longrightarrow
\nu_* \Omega_{\overline{X}} \hookrightarrow
 \omega^R_{X}.
$$
In  \cite[Chap. VIII]{AK70} it  is proved that $\omega_{X} \stackrel{\phi}{\cong} \omega_{X}^R$ and
$d_R=\phi d$, where $d: \OX0 \longrightarrow  \omega_{X}$ is the map defined in the  Section 1.
Since $\OX0$ is a one-dimensional reduced ring we know that $\omega_{(X,0)}$
is a sub-$\OX0$-module of $\mathrm{tot}(\OX0)$, \cite[3.3.18]{BH97}.
There is a the perfect pairing, \cite[Chapter IV]{Ser59},
$$
\begin{array}{ccccc}
  \frac{\nu_* {\mathcal O}_{\overline{X}}}{\OX0} & \times& \frac{\omega_{(X,0)}}{\nu_* \Omega_{\overline{X}}} & \stackrel{\eta}{\longrightarrow}  & \mathbb C \\
  F & \times  & \alpha & \longrightarrow  & {\mathrm{res}}_{p}(F \alpha)
\end{array}
$$
notice that for all $\lambda\in R$ it holds
$
\eta(\lambda F, \alpha)= {\mathrm{res}}_{p}(\lambda F \alpha)=\eta( F, \lambda\alpha).
$

\medskip
\begin{proposition}
\label{canon}
Let $X$ be a branch of $(\res
^n,0)$ and $\overline{X}$ its normalization.
Then  we have an isomorphism of the $\delta(X)$ dimensional $\res$-vector spaces:
$$
B^{\perp}\overset{\epsilon}{\cong} \frac{\omega_{X}}{\nu_* \Omega_{\overline{X}}}
$$
such that $\epsilon(g)$ is the coset defined by $\alpha = \sum_{i=0}^{c-1} i! c_{i} t^{-i-1}$, for all $g=\sum_{i=0}^{c-1} c_i u^i\in B^{\perp}$.
\end{proposition}
\begin{proof}
We write $B=\OX0$,  $\Gamma =\nu_* {\mathcal O}_{\overline{X}}$, and
$\Omega_{\overline{X}}= \Gamma dt$.
Then $\epsilon$ is the composition of the isomorphisms induced by the above two perfect pairings
$$
B^{\perp}
\overset{\epsilon_1}{\cong} \left(\frac{\Gamma}{B}\right)^*
\overset{\epsilon_2}{\cong} \frac{\omega_{X}}{\nu_* \Omega_{\overline{X}}}
$$
Next we describe both morphims $\epsilon_1, \epsilon_2$.
Given $g\in B^{\perp}$ we can write it as
$$
g= c_0+ c_1 u+\dots, c_{c-1} u^{c-1},
$$
so $\epsilon_1(g)$ is the linear form induced by
$\xi:\Gamma^*\longrightarrow \res$ defined by:
if  $f=\sum_{i\ge 0} a_i t^i\in \Gamma$ then
$$
\xi(f)=\sum_{i=0}^{c-1} i! a_i c_i.
$$

On the other hand, every $\alpha \in \omega_X$ can be written as $\alpha =t^n h(t) dt$ with $n\in \mathbb Z$
and $h(t)\in \Gamma$ an invertible series.
From \cite[Proposition 2.6]{Eli16} we get that $\alpha = \sum_{i\ge-c} e_i t^i$
such that ${\rm{res}}_0(\alpha F)=0$ for all $f\in B$.
Given $f=\sum_{i\ge 0} a_i t^i\in \Gamma$ we have
$$
{\rm{res}}_0(f \alpha)= \sum_{i=0}^{c-1} a_i e_{-i-1}
$$
so $\epsilon_2^{-1}(\alpha)$ is the linear form iduced by $\xi':\Gamma^*\longrightarrow \res$ defined by:
$$
\xi'(f)=\sum_{i=0}^{c-1} a_i e_{-i-1}.
$$
From this we deduce that $e_{-i-1}=i! c_i$ for $i=0,\dots, c-1$.
\end{proof}

\medskip
\begin{example}[Example 2.7, \cite{Eli16}]
\label{exp-example}
Let us consider the monomial curve $X$ with parametrization $x_1=t^4, x_2=t^7, x_3=t^9$.
We have $c=11$, $\delta= 6$.
Then $\omega_X$ is the $\res$-vector space   spanned  by
$t^{-11}, t^{-7}, t^{-6}, t^{-4}, t^{-3}, t^{-2}, t^n, n\ge 0$, and the quotient
$\omega_X/ \nu_* \Omega_{\overline{X}}$ admits as $\res$-vector space base the cosets of
$t^{-11}, t^{-7}, t^{-6}, t^{-4}, t^{-3},$ $  t^{-2}$, and $\OX0^{\perp}$ is the $\res$-vector space with basis
$u, u^2, u^3, u^5, u^6, u^{10}$.
\end{example}


\medskip
\baselineskip=10pt

\providecommand{\bysame}{\leavevmode\hbox to3em{\hrulefill}\thinspace}
\providecommand{\MR}{\relax\ifhmode\unskip\space\fi MR }
\providecommand{\MRhref}[2]{%
  \href{http://www.ams.org/mathscinet-getitem?mr=#1}{#2}
}
\providecommand{\href}[2]{#2}

\end{document}